\documentclass{article}
\usepackage{cmap}
\usepackage{amsthm}
\usepackage[cp1251]{inputenc}
\usepackage[T2A]{fontenc}
\usepackage[english]{babel}
\usepackage{amsfonts}
\usepackage{amssymb}
\usepackage{amsmath}
\usepackage{color}
\usepackage{hyperref}
\usepackage[numbers,sort&compress]{natbib}
\usepackage[left=2.5cm,right=2.5cm, top=2.5cm,bottom=3cm]{geometry}
\newtheorem{lemma}{Lemma}
\newtheorem{prop}{Proposition}
\newtheorem*{prop*}{Proposition}

\newtheorem*{theorem}{Theorem}

\newtheorem*{KLconj}{Kervaire-Laudenbach Conjecture}
\newtheorem*{maintheorem}{Main theorem}
\theoremstyle{definition}
\newtheorem*{defin}{Definition}
\newtheorem{quest}{Question}
\newtheorem{rem}{Remark}

\let\tilde\widetilde

\begin{document}

\title{Unimodular equations which
do not preserve the derived length
of a group\footnote{
This research was supported by the Russian Science Foundation under grant no. 22-11-00075}
}
\author{Mikhail A. Mikheenko
\\{\small
Faculty of Mechanics and Mathematics
of Lomonosov Moscow State University
}
\\{\small
Moscow Center for
Fundamental and Applied Mathematics}\\
{\normalsize mamikheenko@mail.ru}
}
\date{}
\maketitle

\begin{abstract}
It is a known fact that
any unimodular equation over an abelian group
has a solution in that group itself. It is also known that
for metabelian groups
this does not hold; moreover, there is a unimodular
equation over some metabelian group which
has no solutions in any larger metabelian group.
Here we present the proof of an analagous fact
for solvable groups of higher derived lengths.
\end{abstract}

\section{Introduction}

In this paper we adopt the following notation for
commutator and conjugation:
$g^h = h^{-1}gh, [g,h] = g^{-1}h^{-1}gh$.
An infinite cyclic group with the generator $g$
is denoted by $\langle g \rangle_\infty$, whereas
the cyclic group of order $n$ is denoted by $\langle g \rangle_n$.

$G'$ denotes the commutator subgroup of $G$,
while $G^{(n)}$ denotes the $n$th commutator subgroup of $G$.
The semidirect product of groups $B$ and $A$
with $B$ as the normal subgroup
is denoted by $B \leftthreetimes A$.

Here the direct wreath product $B\wr A$
of $B$ and $A$ is understood as
the semidirect product
$\left( \underset{a\in A}{\times} B_a \right) \leftthreetimes A$,
where $\underset{a\in A}{\times} B_a$ is
the direct product of isomorphic copies $B_a$
of $B$ marked by elements $a \in A$
and $A$ acts on the components of this product
as follows:
$\left( b_a \right)^{a_1} = b_{aa_1}$, where $b_a \in B_a, b_{aa_1} \in B_{aa_1}$ are
the elements corresponding to $b\in B$.

The present article is devoted to equations over groups.

\begin{defin}
The equation $w(x)=1$ over the group $G$,
where $w(x) \in G \ast \langle x \rangle_\infty$,
is {\it solvable in a group}
$\tilde G$ if $G \subset \tilde G$
and $\tilde G$ contains a solution to the equation
(i.e. there is an element
$\tilde g \in \tilde G$ such that $w(\tilde g)=1$).
In that case $\tilde G$ is called a {\it solution group}.

If the equation $w=1$ is solvable in some group $\tilde G$,
the equation $w=1$ over $G$ is called {\it solvable}.

\end{defin}

The solvability of a (finite or infinite) system of equations
in multiple variables over a group is defined likewise.
Equations and systems of equations over groups is a classical field of study,
see
\cite{Sh67, How81, B80, B84, K06, ABA21, EH91, BE18, EH21,
T18, NT22, Sh81, Le62, Kr85, KT17, KM23, KMR24, K93, KP95,
M24, M25, EdJu00, IK00, GR62, G83}.
There is a survey on equations over groups as well, see \cite{Ro12}.

The present article is concerned with unimodular equations.

\begin{defin}
An equation
$
w(x)=1
$
over a group $G$
is called
{\it unimodular},
if the exponent sum of $x$
in the word $w(x)$ equals $\pm1$.
\end{defin}

For example, the equation $xg_1x^{2}g_2x^{-4}=1$ is unimodular,
while the equations $x^2g_1x^{-4}=1$ and $x^{-1}g_1xg_2=1$
fail to be such.

Groups over which all unimodular equations are solvable
form a wide class, including:
\begin{itemize}
\item finite groups \cite{GR62}; 
\item residually finite groups
(it follows from the previous item);
\item locally residually finite groups, e.g. locally finite groups
(which also follows from the first item);
\item locally indicable groups
\cite{B80, Sh81, B84, How81}, i.e. groups
whose every non-trivial finitely generated subgroup
admits a surjective homomorphism onto
an infinite cyclic group;
\item hyperlinear groups \cite{P08} (it is unknown if this class coincides with
the class of all groups);
\item torsion-free groups \cite{K93}.
\end{itemize}

In fact, for most of these subclasses unimodular equations are not the only
ones which are proved to be solvable in mentioned papers. For
complete statements, see the articles cited above.

Also, there is a following conjecture on solvability
of unimodular equations over groups.
\begin{KLconj}
Any unimodular equation over any group is solvable.
\end{KLconj}
It is still yet unknown whether the conjecture is true.

Apart from seeking a solution of an equation over a group $G$
in any group $\tilde G$, no matter with which properties,
one can also ask if a solution can be found
in a group similar to the initial group.

For example, a unimodular equation $w(x)=1$
over an abelian group $A$ always has a solution in
$A$ itself
(it is an easy exercise if one notes that
since the solution is searched in an abelian group,
the variable can be assumed to commute with the coefficients).
It is in fact even true for nilpotent groups instead of abelian ones:
this is a result of A.L. Shmel'kin. Here is a very weakened
statement of that result:
\begin{theorem}[\cite{Sh67}]
Suppose that $G$ is a nilpotent group.
Then any unimodular equation over
$G$ has a solution in $G$.
\end{theorem}

Some of the earlier mentioned subclasses of groups
also admit solving unimodular equations over a group
in a group from the same subclass.
For instance, a unimodular equation over a finite group
is solvable in a finite group as well \cite{GR62},
and any unimodular equation over a locally indicable group
has a solution in a locally indicable group
\cite{How82}.

The situation is more complicated for the class
of solvable groups.
\begin{prop}[\cite{KMR24}]\label{BadMetab}
There is a unimodular equation $w(x)=1$
over a metabelian group $G$
such that the equation has no solutions in metabelian groups.

Moreover, $G$ can be chosen both to be finite
(of order $42$) or to be torsion free (and such that any unimodular
equation over $G$ has a solution in solvable groups
of derived length 3).
\end{prop}
Both groups $G$ from this proposition
have elements $a,c \in G$ such that
$a^6, c \in G'$, but $cc^ac^{-a^3}c^{-a^4}\neq1$.
Unimodular equations over them which have no solutions
in metabelian groups have the form $xx^{-a}x^{a^2}=c$.
The group of order $42$ is the semidirect product
$\langle c \rangle_7 \leftthreetimes \langle a \rangle_6$,
with $c^a = c^5$. Indeed, in this group
$$
c c^a c^{-a^3} c ^{-a^4}=c^{1+5-125-625} = c^{(1+5)(1-125)} = c^{6\cdot(-124)} \neq 1,
$$
as both $6$ and $124$ are not divisible by $7$.

It is yet unknown if the equation
$xx^{-a}x^{a^2}=c$ over
$\langle c \rangle_7 \leftthreetimes \langle a \rangle_6$
has a solution in somesolvable group.

In \cite{M24}, $\langle c \rangle_7 \leftthreetimes \langle a \rangle_6$
was proved to be
the minimal example of a metabelian group,
for which the proposition above holds. 

However, it was still unknown if there was an analagous
example of a solvable group with a unimodular equation
over it, which has no solutions in solvable groups
of the same derived length, 
for an arbitrary derived length (except 1, of course).

Here we fill this gap.
\begin{maintheorem}
For any natural number
$n\geqslant 2$
there is a unimodular equation over a solvable group of derived length
$n$ which has no solutions in solvable groups of the same derived length.
\end{maintheorem}

The author thanks the Theoretical Physics and Mathematics
Advancement Foundation
``BASIS''.
The author also thanks Anton A. Klyachko
for the more clear proof of Lemma \ref{secondcomm}
which was proposed on a seminar.

\section{Preliminaries}

\begin{lemma}\label{cneq}
Suppose that $G$ is a metabelian group with elements $a,c \in G$
such that $a^6, c \in G'$.
Also, suppose that $c^a \neq c^3$
(in particular, $c \neq 1$).
Then $c \neq cc^ac^{-a^3}c^{-a^4}$.
\end{lemma}

\begin{proof}
Assume $c = c c^a c^{-a^3}c^{-a^4}$.
Note that in this case $c c^a c^{-a^3}c^{-a^4} \neq 1$ and
$$
c^{a^3} = c^{a^3} c^{a^4} c^{-1} c^{-a}  =
\left( c^a c c^{-a^4} c^{-a^3} \right)^{-1}
=\left( c c^a c^{-a^3}c^{-a^4}  \right)^{-1} = c^{-1},
$$
where the third equation holds because $c \in G'$
commutes with its conjugates.
Hence $c=cc^ac^{-a^3}c^{-a^4} = c^2 c^{2a}$.
Then $c = c^{-2a}=c^{4a^2}=c^{-8a^3}=c^{8}$.
Now it follows that:
\begin{itemize}
\item $c^7=1$. As $c\neq 1$, we get that the order of $c$ is $7$.
\item $\langle a \rangle$ acts on $\langle c \rangle_7$ by conjugation,
and $a^6$ acts trivially.
\end{itemize}

Consider the action of $a$ on $\langle c \rangle_7$.
This action can not be an automorphism of order $1$, $2$ or $3$,
since in each of these cases $cc^ac^{-a^3}c^{-a^4} = 1$.
That is why $a$ acts by an automorphism of order $6$.
So this automorphism is either $c \mapsto c^3$ or $c \mapsto c^5$.
But it can not be $c \mapsto c^3$ by condition, and for
$c \mapsto c^5$ we have
$$
cc^ac^{-a^3}c^{-a^4} = c^{(1+5)\cdot(1-125)} = c^{6 \cdot 2} = c^5 \neq c.
$$
We get a contradiction, hence the assumption
$c = c c^a c^{-a^3} c^{-a^4}$ is false.
\end{proof}

\begin{lemma}\label{secondcomm}
Suppose that $G\ni a,c$ is a metabelian group such that
$a^6, c \in G'$ and
$\tilde G$ is a solution group for the equation
$x x^{-a} x^{a^2} = c$.
Then $c c^a c^{-a^3} c ^{-a^4} \in \left(\tilde G\right)''$.
\end{lemma}

\begin{proof}
Notice that $x \in \left( \tilde G \right)'$ $\biggl($as it equals $c$ modulo
$\left( \tilde G \right)'$$\biggr)$. So, $x^a, x^{a^2}, \ldots, x^{a^5}$
lie in $\left( \tilde G \right)'$ as well.

Now consider $\tilde G$
modulo its second commutator subgroup $\left( \tilde G \right)''$.
In that case the elements of $\left(\tilde G\right)'$, including
$x,x^{a},\ldots,x^{a^5}$, can be assumed to commute with each other.
Therefore we adopt the following notation:
$$
h^{a+b} = h^ah^b
$$
for $h \in \left( \tilde G\right)', a,b \in \tilde G$.
Consider the element
$g = c c^a c^{-a^3} c ^{-a^4}$. Substitute
$c$ with $x^{1-a+a^2}$ there.
In our notation we get
$$
g = c^{1+a-a^3-a^4} = \left( x^{1-a+a^2} \right)^{1+a-a^3-a^4}
=
x^{(1-a+a^2)(1+a-a^3-a^4)} = x^{(1-a+a^2)(1+a)(1-a^3)}=x^{1-a^6}=1.
$$
Thus, $g = 1$ modulo $\left( \tilde G \right)''$.
This means that $g\in \left(\tilde G \right)''$, as needed.

\end{proof}

\begin{rem}
In particular, if $c c^a c^{-a^3} c ^{-a^4} \neq 1$,
then the equation $xx^{-a}x^{a^2}=c$ over
$G$ from the condition of the lemma
is unsolvable in metabelian groups.
\end{rem}

\section{The proof of the main theorem}

Remind the statement of the main theorem.

\begin{maintheorem}
For any natural number
$n\geqslant 2$
there is a unimodular equation over a solvable group of derived length
$n$ which has no solutions in solvable groups of the same derived length.
\end{maintheorem}

\begin{proof}
Consider a group $H = B \wr G$, where $B$ is an arbitrary
solvable group
of derived length exactly $n-2$, whereas
$G\ni a,c$ is a metabelian group such that
$a^6, c \in G'$ but $cc^{a}c^{-a^3}c^{-a^4} \neq 1$.
We can choose $G$ as one of the groups from Proposition \ref{BadMetab},
e.g. $\langle c \rangle_7 \leftthreetimes \langle a \rangle_6$
with $c^a = c^5$.
Consider the equation $x x^{-a} x^{a^2} = c$ over $H$,
where $a,c$ are the elements of the subgroup
$G$ corresponding to the lemma.

Suppose that $\tilde H$ is a solution group for this equation.
Consider $\left( \tilde H \right)''$.
\begin{itemize}
\item On the one hand, $\left( \tilde H \right)''$ contains the elements of the form
$b b^{-c}b^{-a} b^{ac}$, where $b \in B$
(notice that $cc^{a}c^{-a^3}c^{-a^4} \neq 1$ implies $a \notin G'$ и $c \neq 1$;
so, $1,c,a,ac$ are four distinct elements of $G$).
Indeed, $H' \subset \left( \tilde H \right)'$ contains
$[b,a] = b^{-1} b^a$.
Therefore, $H'' \subset \left( \tilde H \right)''$ contains
$$
[b^{-1}b^a , c] = b^{-a} b b^{-c} b^{ac} = b b^{-c} b^{-a} b^{ac},
$$
as the different copies of $B$ commute with each other.
For convenience, write this element as $h = b^{1-c-a+ac}$.

\item On the other hand, Lemma \ref{secondcomm} states that
$\left( \tilde H \right)''$ also contains $g=cc^ac^{-a^3}c^{-a^4}\neq 1$.
\end{itemize}
Then consider the commutator
$[g^{-1},h] = b^{(1-c-a+ac)(1-g^{-1})}$.
It is an element of the direct product of copies of $B$.
We state that its coordinate coresponding to $1 \in G$ equals $b$.
For this to be true, it is enough to prove that
the term $1$
appears
in the expression
$(1-c-a+ac)(1-g^{-1}) \in \mathbb ZG$
only once.
Notice that the term with the element $a$ does not equal
$1$, as $a$ does not lie in the commutator subgroup of $G$
(or else $cc^ac^{-a^3}c^{-a^4}$ equals $1$), while $g^{-1}$ and $c$ lie there.
Also note that $c \neq 1 \neq g^{-1}$. It remains to understand that $cg^{-1} \neq 1$,
or, in other words, that $c \neq g$, which is true by Lemma \ref{cneq}
since $c^a = c^5 \neq c^3$.

As a result we get that the third commutator subgroup of $\tilde H$
contains the elements $[g^{-1}, h]$ of the direct product of copies of
$B$, and their coordinates corresponding to $1 \in G$ run over all the group $B$.
Therefore the subgroup generated by those elements
lies in the $3$rd commutator subgroup of $\tilde H$
and has a nontrivial
$\left( n - 3 \right)$th commutator subgroup.
Thus, $\left (\tilde H\right)^{(n)}\neq \{1\}$, as needed.
\end{proof}

\begin{rem}
Note that since the group $G$ (and the group $B$)
can be chosen either to be finite or to be torsion free,
the group $H$ can be chosen to satisfy one of these properties as well.
\end{rem}

\section{Further questions and surmises}

The following question remains open.

\begin{quest}\label{questmetab}
Does the equation
$xx^{-a}x^{a^2}=c$
over
$\langle c \rangle_7
\leftthreetimes
\langle a \rangle_6$,
where $c^a=c^5$,
have a solution in some solvable group?
\end{quest}

The following proposition gives a slight hope
for the negative answer.

\begin{prop} \label{hope}
Suppose that $G$ is a metabelian group with the elements $a,c \in G$
such that $a^6, c \in G'$ and $cc^ac^{-a^3}c^{-a^4}\neq 1$.
Denote $g = c c^a c^{-a^3} c^{-a^4}$.
If $c$ has a finite order coprime to $6$,
then $g g^a g^{-a^3} g^{-a^4} \neq 1$
and $g$ has a finite order coprime to $6$. %
\end{prop}

\begin{proof}
Note that if the order of $c$ is $n$, then
$g^n = 1$, and the order of $g$ divides $n$.
Particularly, if $n$ is coprime to $6$, then the order of
$g$ is also coprime to $6$ (and is not equal to $1$ by condition).

Then notice that $g g^{-a} g^{a^2} = c^{(1-a+a^2)(1+a)(1-a^3)}=1$.
Also notice that $g^{a^3} = g^{-1}$,
and the condition $g g^a g^{-a^3} g^{-a^4} = 1$ can be rewritten as
$g^2 g^{2a} = 1$. Since $g$ has an odd order, it follows that $gg^a=1$, so $g^a=g^{-1}$.
Therefore $gg^{-a}g^{a^2}=g^3=1$.
But this contradicts with the condition of the proposition.
So the assumption $gg^ag^{-a^3}g^{-a^4}=1$ is false.
\end{proof}

\begin{rem}
If the order of $c$ is even or divisible by $3$, then
$gg^ag^{-a^3}g^{-a^4}$ can be equal to $1$.
To see this, consider the group
$\left( \langle b \rangle_n \times \langle d \rangle_n \right)\leftthreetimes \langle a \rangle_6$
with the action $b^a = bd; d^a = b^{-1}$.
The automorphism acts as follows:
$$
b \mapsto bd \mapsto d \mapsto b^{-1} \mapsto b^{-1}d^{-1} \mapsto d^{-1} \mapsto b.
$$
Take $c = b = \left[ b^{-1}d^{-1}, a \right]$.
Then
$$
g = b \cdot bd \cdot b \cdot bd = b^4d^2,
$$
$$
gg^ag^{-a^3}g^{-a^4} = b^4d^2 \cdot b^2d^4 \cdot b^4d^2 \cdot b^2d^4 = b^{12}d^{12}. 
$$
Now we see that if $n$ equals $3$ or $4$, then $g\neq 1$, but
$gg^ag^{-a^3}g^{-a^4}=1$.
\end{rem}
Now if for the elements $a, c$ from Proposition
\ref{hope} there is an element $y$ of a solution group $\tilde G$ of equation $xx^{-a}x^{a^2} = c$
such that $yy^{-a}y^{a^2}=g$,
then such group, possibly, will have non-trivial fourth
commutator subgroup as well as the third one.
Likewise, $\tilde G$ would have a non-trivial
fifth commutator subgroup an so on, in other words,
$\tilde G$ will fail to be solvable.
However, it is unclear if such element $y$ exists and
whether the nontriviality of the fourth commutator subgroup
actually follows from the existence of $y$.
In other words, the question needs further consideration.

The following question is connected to Question \ref{questmetab}.
\begin{quest} \label{questsolv}
Is it true that any unimodular equation over a solvable group
has a solution in a solvable group (maybe of larger derived length) as well?
\end{quest}
If the answer to Question \ref{questmetab} is negative,
then the answer to Question \ref{questsolv}
is negative as well.
And if the answer to Question \ref{questmetab}
is positive, then, probably, the Question \ref{questsolv}
will have the positive answer for the groups from the proof of the main theorem.

\bibliography{biblioeng}

\end{document}